\documentclass{amsart}

\usepackage{etex}
\usepackage{graphicx,amssymb}
\usepackage{color}
\usepackage{amsmath}
\usepackage{amsthm}
\usepackage{epsfig}
\usepackage{pstricks}

\usepackage{tikz}
\usepackage{soul}   
\usepackage{xy} \input xy \xyoption{all}
\usepackage{algorithm}
\usepackage{algorithmic}
\usepackage{color}
\usepackage{cancel}
\usepackage{caption}
\usepackage{subcaption}
\usepackage{enumitem}

\newtheorem{theorem}{Theorem}[section]
\newtheorem{lemma}[theorem]{Lemma}

\theoremstyle{definition}

\theoremstyle{remark}
\newtheorem{remark}[theorem]{Remark}

\numberwithin{equation}{section}

\def\qed{\hfill\vbox{\hrule\hbox{\vrule\kern3pt\vbox{\kern6pt}\kern3pt\vrule}\hrule}\bigskip}

\def\P{{\mathbb{P}}}

\def\A{{\mathbb{A}}}

\def\mO{{\mathcal{O}}}

\makeatletter
\newcommand{\dashedrightarrow}[1][2pt]{%
  \settowidth{\@tempdima}{$\longrightarrow$}\longrightarrow
  \makebox[-\@tempdima]{\hskip-1.5ex\color{white}\rule[0.5ex]{#1}{1pt}}
  \phantom{\longrightarrow}
}
\makeatother

\newcounter{mnotecounter}

\title{A note about rational surfaces as unions of affine planes}

\author{Jorge Caravantes$^1$,  J. Rafael Sendra$^1$,  David Sevilla$^2$ and Carlos Villarino$^1$}

\address{$^1$ Universidad de Alcal\'a\\ Dpto. de F\'isica y Matem\'aticas \\ 28871 Alcal\'a de Henares (Madrid), Spain \\ \vskip 1pt
$^2$  Centro U. de M\'erida\\ Universidad de Extremadura \\ Av. Santa Teresa de Jornet 38, 06800 M\'erida (Badajoz), Spain}

\thanks{The authors are partially supported by the grant PID2020-113192GB-I00/AEI/ 10.13039/501100011033 (Mathematical Visualization: Foundations, Algorithms and Applications) from the Spanish State Research Agency (Ministerio de Ciencia e Innovación). J. Caravantes and C. Villarino belong to the Research Group ASYNACS (Ref. CT-CE2019/683). D. Sevilla is a member of the research group GADAC and is partially supported by Junta de Extremadura and Fondo Europeo de Desarrollo Regional (GR21055).}


\begin{document}

\begin{abstract}
We prove that any smooth rational projective surface over the field of complex numbers has an open covering consisting of 3 subsets isomorphic to affine planes.
\end{abstract}

\maketitle

Since all smooth rational curves are isomorphic to $\P^1$, they can be seen as the union of two affine lines. In dimension two, as a consequence of the structure Theorem \ref{tma:minimal} below, all rational surfaces admit a covering of open subsets isomorphic to the affine plane. However, up to the authors' knowledge, no general results are known on the minimal number of open subsets of such a covering, while some advances are known by computer algebrists in terms of surjectivity of parametrizations \cite{BajajRoyappa1995a,SendraSevillaVillarino2016a, CaravantesSendraSevillaVillarino2018,CaravantesSendraSevillaVillarino2021}.
In this short note we prove that all projective smooth rational surfaces behave like the projective plane in this aspect.

\section{Main result}\label{sec:Hirzebruch}

\begin{theorem}\label{tma:gordo}
Let $X$ be a projective smooth rational surface over the complex field. Then, there are three open subsets $U_0,U_1,U_2\subset X$ such that:
\begin{enumerate}
\item $U_0\cup U_1\cup U_2 =X$.
\item For all $i=0,1,2$, $U_i$ is isomorphic to the affine plane.
\end{enumerate}
\end{theorem}

To prove Theorem \ref{tma:gordo}, we will use the following well-known result:

\begin{theorem}(see e.g. \cite[Theorem V.10]{Beauville1996a})\label{tma:minimal}
Every non-singular rational surface can be obtained by repeatedly blowing up either $\P^2$ or the projective bundle $\P(\mathcal{O}_{\P^1}\oplus\mathcal{O}_{\P^1}(-n))$ (the Hirzebruch surface $\Sigma_n$), for $n\ne 1$.
\end{theorem}

By Theorem \ref{tma:minimal}, there exists a chain of morphisms $\pi=\pi_1\circ\cdots\circ\pi_r:X\to M$ such that $M$ is either $\P^2$ or a Hirzebruch surface and $\pi_i:X_i\to X_{i-1}$ is the blowup of a smooth surface at a single point. Let $E$ be the exceptional divisor of $\pi$ and $E_i$ the exceptional divisor of $\pi_i$. Then, $\pi(E)\subset M$ is a finite set of closed points and $\pi_i(E_i)$ is one closed point. Moreover, $E_i\simeq\P^1$ and $E$ is a finite union of smooth rational curves (in fact, $E_1$ and the proper transforms of all the $E_2,...,E_r$). We begin by proving Theorem \ref{tma:gordo} for $X=M$ with care for the centers of the blowups:

\begin{lemma}\label{lma:caso_base}
In the above conditions, there exist three open subsets $U_0^0$, $U_1^0$, $U_2^0$ such that:
\begin{enumerate}
\item $U_0^0\cup U_1^0\cup U_2^0 =M$.
\item For all $i=0,1,2$, $U_i$ is isomorphic to the affine plane.
\item $\pi(E)\subset U_0^0\cap U_1^0\cap U_2^0$.
\end{enumerate}
\end{lemma}

\begin{proof}
The case $M=\P^2$ is well-known. Since $\pi(E)$ is finite and we work over an infinite field, one can choose three different projective lines $L_1$, $L_2$ and $L_3$ in $\P^2$ such that $\pi(E)\cap(L_1\cup L_2\cup L_3)=\emptyset=L_1\cap L_2\cap L_3$. 

If $M$ is a Hirzebruch surface, then it is the projective bundle of a rank two vector bundle $\mO_{\P^1}\oplus\mO_{\P^1}(-m)$ over $\P^1$. This means that there is a surjective morphism $p:M\to \P^1$ such that, for any point $P\in\P^1$, $p^{-1}(\P^1-\{P\})\simeq (\P^1-\{P\})\times\P^1$. Then, since we work over an infinite field, one can choose a closed point $P_0\in \P^1-p(\pi(E))$ with its isomorphism $q_0:p^{-1}(\P^1-\{P_0\})\to \A^1\times\P^1$. Then, we choose a line $L_0=\A^1\times\{Q_0\}$, such that $q_0(\pi(E))\cap L_0$ is empty. With this choice, $U_0^0=q_0^{-1}(\A^1\times\P^1-L_0)$ is isomorphic to $\A^2$ and contains $\pi(E)$.

Then, $M-U_0^0$ is the union of two rational curves $C_1:=p^{-1}(P_0)$ and $C_2:=\overline{q_0^{-1}(L_0)}$. Choosing $P_1\in\P^1-(p(\pi(E))\cup \{P_0\})$ (again, the complement of a finite set), together with the isomorphism $q_1:p^{-1}(\P^1-\{P_1\})\to \A^1\times\P^1$, we have that $p^{-1}(\P^1-\{P_1\})$ contains $C_1$ and $C_2$ with the exception of the point $R_1:=C_2\cap p^{-1}(P_1)$ (the intersection of a section $\A^1\to\A^1\times\P^1$  with a fiber). We now choose a line $L_1=\A^1\times\{Q_1\}$ such that:
\begin{itemize}
\item  $Q_1\in\P^1$ is not in the second projection of $q_1(\pi(E))\in\A^1\times P^1$; and
\item $L_1\neq q_1(C_2)$ (i.e. we are asking a constant section not to coincide with a given one, which is an open condition for $Q_1$), so the intersection of the two curves is finite.
\end{itemize}  
 Then, $U_1^0=q_1^{-1}(\A^1\times\P^1-L_1)$ is isomorphic to $\A^2$ and contains $\pi(E)$.

Now, $M-(U_0^0\cup U_1^0)=(C_1\cup C_2)-U_1^0$ is the finite set $A:=\{R_1\}\cup q_1^{-1}(L_1\cap q_1(C_2))$. Finally, we have again the complement of a finite set to choose $P_2\in\P^1-(p(\pi(E))\cup p(A)\cup \{P_0,P_1\})$ with the isomorphism $q_2:p^{-1}(\P^1-\{P_2\})\to \A^1\times\P^1$, so we have $A\subset p^{-1}(\P^1-\{P_2\})$. We now choose $L_2=\A^1\times\{Q_2\}$ such that $Q_2\in \P^1$ is not in the second projection of the finite set $q_2(A\cup\pi(E))\subset\A^1\times\P^1$, and we define $U_2^0=q_2^{-1}(\A^1\times\P^1-L_2)\simeq\A^2$. Then $\pi(E)\subset U_2^0$ and, since $A\subset U_2^0$, we have that $U_0^0\cup U_1^0\cup U_2^0 =M$.
\end{proof}


\begin{remark}\label{rem:explotar_el_plano}
Let Bl$_P(\A^2)$ be the blowup of the affine plane at a point $P$. Then, there exists one morphism $\pi_l:U_l\simeq\A^2\mapsto\A^2$ for each affine line $l\subset\A^2$ passing through $P$, such that:
\begin{enumerate}
\item for $l_1\ne l_2$, $U_{l_1}\cup U_{l_2}=$Bl$_P(\A^2)$. 
\item if $E_{\A^2}$ is the exceptional divisor of Bl$_P(\A^2)$, for any line $l$ passing through $p$, $E_{\A^2}-U_{l}$ consists in one point, given by the isomorphism between $E_{\A^2}$ and the $\P^1$ of all lines through $P$.
\item the restriction $\pi_l|_{U_l-E_{\A^2}}$ is an isomorphism between $U_l-E_{\A^2}$ and $\A^2-l$.
\end{enumerate}
\end{remark}

\begin{lemma}\label{lma:paso_inductivo}
Let $X$ be a smooth rational surface such that there exist three open subsets $U_0,U_1,U_2\subset X$ with 
\begin{enumerate}
\item $U_0\cup U_1\cup U_2 =X$.
\item For all $i=0,1,2$, $U_i$ is isomorphic to the affine plane.
\end{enumerate}
Consider a finite set $A_1\subset U_0\cap U_1\cap U_2$. Let $P\in (U_0\cap U_1\cap U_2)-A_1$ be a point and consider $\pi:Y\to X$ to be the blowup of $X$ at $P$. Consider also a finite set $A_2$ in the exceptional divisor $E=\pi^{-1}(P)\subset Y$. Then, there are three open subsets $U_0',U_1',U_2'\subset Y$ such that 
\begin{enumerate}
\item $U_0'\cup U_1'\cup U_2' =Y$.
\item For all $i=0,1,2$, $U_i'$ is isomorphic to the affine plane.
\item Both $A_2$ and the proper transform of $A_1$ are contained in $U_0'\cap U_1'\cap U_2'$
\end{enumerate}
\end{lemma}

\begin{remark}\label{rem:complemento_finito}
In the conditions of Lemma \ref{lma:paso_inductivo}, note that for any $i,j=0,1,2$, $i\ne j$, $X-(U_i\cup U_j)$ is a Zariski closed subset of a projective surface which is contained in $U_k\simeq \A^2$, with $i\ne k\ne j$. Since it is a projective scheme in an affine space, it must be finite.
\end{remark}

\begin{proof}
Taking into account Remark \ref{rem:explotar_el_plano}, consider a line $l_0\subset U_0\simeq\A^2$ through $P$ such that 
\begin{itemize}
\item The intersection of $l_0$ with the finite set $X-(U_1\cup U_2)$ (see Remark \ref{rem:complemento_finito}) is empty.
\item $A_1\cap l_0=\emptyset$.
\item The intersection point of the proper transform of $l_0$ with the exceptional divisor is not in $A_2$.
\end{itemize}
 Then, we define $U_0'$ to be the open subset $U_{l_0}$ of the blowup of $U_0$ at $P$. $U_0$ is isomorphic to the affine plane, as said in Remark \ref{rem:explotar_el_plano}, and  $Y-U_0'$ consists in the proper transform of $l_0\cup (X-U_0)$. Therefore, it is one-dimensional.

Now, we choose a line $l_1\subset U_1\simeq\A^2$ such that the following open conditions are satisfied:
\begin{enumerate}
\item The intersection of $l_1$ with the finite set $X-(U_0\cup U_2)$ (see Remark \ref{rem:complemento_finito}) is empty.
\item the intersection multiplicity of $l_1$ and $l_0$ at $P$ is 1 (note that $l_1$ is smooth at $P$, so we are asking that $l_1$ is not the tangent line at $P$ to $l_0$, when we see them in $U_1$).
\item $l_1$ does not contain any point in $l_0\cap (X-U_2)$ (note that $l_0$ is irreducible and $P\subset l_0\cap U_2$, so such intersection is finite).
\item $A_1\cap l_1=\emptyset$.
\item The intersection point of the proper transform of $l_1$ with the exceptional divisor is not in $A_2$.
\end{enumerate}
Since we work over an infinite field, these conditions define a nonempty Zariski open subset to choose $l_1$ from. Now, we define $U_1'$ to be $U_{l_1}\simeq\A^2$. The whole exceptional divisor is in $U_0'\cup U_1'$. Then, $Y-(U_0'\cup U_1')$ is the proper transform of the finite sets $B_1=[l_0\cup (X-U_0)]\cap l_1$ and $B_2=X-(U_0\cup U_1)$.

Note that $P\not\in B_1\cup B_2\subset U_2$, so we choose a last line $l_2\subset U_2\simeq\A^2$ such that
\begin{itemize}
\item  the intersection of $l_2$ with the finite set $X-(U_0\cup U_1)$ (see Remark \ref{rem:complemento_finito}) is empty,
\item $(A_1\cup B_1\cup B_2)\cap l_2=\emptyset$, and
\item  the intersection point of the proper transform of $l_2$ with the exceptional divisor is not in $A_2$.
\end{itemize}  
Defining $U_2'=U_{l_2}$, one concludes the proof.
\end{proof}

\begin{remark}\label{rem:induction_high_dim}
It is likely that a generalisation of Lemma \ref{lma:paso_inductivo} to higher dimension is possible. However, it is not yet known if all rational varieties of dimension greater than 2 are covered by open subsets isomorphic to open subsets of $\A^n$ (see \cite{Gromov1989} for the original question). These varieties are known as plain \cite{Bodnar2008} or uniformly rational \cite{BogomolovBohning2013} and it is possible that the main result can be extended to higher dimension for this type of varieties. 
\end{remark}

\begin{proof}{\it (of Theorem \ref{tma:gordo})} 
By Lemma \ref{lma:caso_base}, we have that $M=X_0=U_0^0\cup U_1^0\cup U_2^0$ with $U_i^0\simeq\A^2$ and $\pi(E)\subset U_0^0\cap U_1^0\cap U_2^0$. Now we apply Lemma \ref{lma:paso_inductivo} to $\pi_i:X_i\to X_{i-1}$, choosing 
\[A_1=\left[\pi_{i}\circ\pi_{i+1}(E_{i+1})\cup\cdots\cup\pi_i\circ\cdots\circ\pi_r(E_r)\right]-\{P_i\}\]
(i.e. the points to be the center of future blowups outside $\{P_i\}$) and 
\[A_2=\left[\pi_{i+1}(E_{i+1})\cup\cdots\cup\pi_{i+1}\circ\cdots\circ\pi_r(E_r)\right]\cap E_i\]
(i.e. the points to be center of future blowups in $E_i$). Note that any curve contracted by $\pi_{i+1}\circ\cdots\circ\pi_i$ is contracted to a point in $\pi_i^{-1}(A_1)\cup A_2$. We then get $U_0^i, U_1^i, U_2^i$ from $U_0^{i-1},U_1^{i-1},U_2^{i-1}$ all isomorphic to $\A^2$ and covering $X_i$, with all centers of future blowups in the intersection of the three open subsets. Then $U_0^r$, $U_1^r$ and $U_2^r$ are the three open subsets in the statement.
\end{proof}


\bibliography{Estructura_Racionales}
\bibliographystyle{alpha}

%
%
%
%
%
%
%
%
%

\end{document}